\documentclass[letterpaper, 10 pt, conference]{ieeeconf}  
\IEEEoverridecommandlockouts                           
\usepackage{amsmath} 
\usepackage{amssymb} 
\newcommand {\matr}[2]{\left[\begin{array}{#1}#2\end{array}\right]}

\usepackage{algorithm}
\usepackage[noend]{algpseudocode}
\usepackage{graphicx}
\graphicspath{{images/}}

\usepackage{tabularx,booktabs}
\usepackage[usenames,dvipsnames]{color}
\definecolor{wheat}{rgb}{0.96,0.87,0.70}
\definecolor{mario}{rgb}{0.8,0.8,1}
\definecolor{themis}{rgb}{0.8,1,0.8}
\definecolor{paolo}{rgb}{0.8,0.8,0.8}
\definecolor{henk}{rgb}{1,0.8,0.8}

\usepackage{accents}
\usepackage{expl3,xparse}

\newcommand{\E}{\mathbb{E}}
\newcommand{\tr}{\mathrm{tr}}

\newtheorem{Theorem}{Theorem}
\newtheorem{Proposition}[Theorem]{Propositon}

\newtheorem{Assumption}[Theorem]{Assumption}
\newtheorem{Definition}[Theorem]{Definition}

\usepackage{tikz}
\usepackage{cite}
\usepackage{soul}

\title{\LARGE \bf {Optimal Scheduling of Downlink Communication for a Multi-Agent System with a Central  Observation Post}}

\author{Mario Zanon, Themistoklis Charalambous, Henk Wymeersch and Paolo Falcone
\thanks{This work was supported by Copplar (project number 32226302), the Swedish Research Council (VR, grant number 2012-4038) and the European Commission Seventh Framework (AdaptIVe, grant number 610428).
M. Zanon, H. Wymeersch and P. Falcone are with the Department of Signals and Systems, Chalmers University of Technology, G\"{o}teborg, Sweden. e-mail: \{name.surname@chalmers.se\}.
T. Charalambous is with the Department of Electrical Engineering and Automation, School of Electrical Engineering, Aalto University, Espoo, Finland. e-mail: \{name.surname@aalto.fi\}.}%
}

\begin{document}

\maketitle
\thispagestyle{empty}
\pagestyle{empty}


%
%
%
%
\begin{abstract}
In this paper, we consider a set of agents, which may receive an observation of their state by a central observation post via a shared wireless network. The aim of this work is to design a scheduling mechanism for the central observation post to decide how to allocate the available communication resources. The problem is tackled in two phases: (i) first, the local controllers are designed so as to stabilise the subsystems for the case of perfect communication; (ii) second, the communication schedule is decided with the aim of maximising the stability of the subsystems. To this end, we formulate an optimisation problem which explicitly minimises the Lyapunov function increase due to communication limitations. We show how the proposed optimisation can be expressed in terms of Value of Information (VoI), 
we prove Lyapunov stability in probability and we test our approach in simulations.
\end{abstract}

%
%
%
%
\section{Introduction}\label{sec:intro}

The advancement of smart devices with increasing \emph{sensing}, \emph{computing} and \emph{control} capabilities makes it possible for several processes to become more intelligent, energy-efficient, safe and secure. However, the components of such systems are spatially distributed and communication between smart devices (being sensors, actuators or controllers) is mainly supported by a \emph{shared, wireless} communication network. These systems are known as Wireless Networked Control Systems (WNCSs); a thorough literature review can be found in \cite{Gupta2010,Zhang2013}. The use of wireless communications in such systems, however, introduces additional challenges due to the limitations imposed by the use of the wireless medium, e.g. packet losses, data rate constraints and interference. Therefore, it is necessary to develop Medium Access Control (MAC) mechanisms for WNCSs.

There are two types of MAC: (a) random access mechanisms, in which each agent can access the network randomly, and (b) scheduling mechanisms, in which a centralised entity decides on the allocation of the resources. Even though random access protocols can easily be implemented in a distributed fashion, it is difficult to provide any performance guarantee.

In this work, we consider a scenario in which several agents may receive an observation of their state by a central observation post, that has the ability to observe the states of all the agents in the system, via a shared wireless network. The shared wireless network is interference-limited, thus allowing only a limited number of simultaneous transmissions. Additionally, these transmissions may not reach the agents since the communication channel can be unreliable. Such setups may arise in several realistic occasions, e.g. when a rescue operation takes place, a drone observes the scene and informs the agents by broadcasting their location with respect to the target mission, one agent at the time. In such cases 
on the one hand the local control schemes are required to stabilise the subsystems (agents); on the other hand, a scheduling mechanism must decide which information is sent to which subsystem. 


For discrete-time linear systems, several scheduling approaches have been proposed in order to make a good use of the communication channel, see e.g.~\cite{Molin2015,Mamduhi2015,Soleymani2016} and references therein. These works
consider the scheduling problem for state estimation of LTI systems with Gaussian noise. With a few exceptions (e.g., \cite{Tomlin2014,Sinopoli2014}), a finite time horizon is considered, in which the problem is a combinatorial optimisation one \cite{Vitus2012}, and hence $\mathcal{NP}$-hard, making the computation of the globally optimal solution over long time horizons computationally expensive.
The works considering the infinite horizon case, focus on estimation only. 

In this paper, we decouple the control and communication allocation problems: (i) first, the local controllers are designed, so as to stabilise the subsystems in the perfect communication case; (ii) second, the communication schedule is decided with the aim of maximising the stability of the subsystems. To this end, we formulate an optimisation problem which explicitly minimises the Lyapunov function increases due to communication limitations. We show how the proposed optimisation can be expressed in terms of VoI which, in the considered case, resembles the concept of cost-to-go used in dynamic programming, and we prove Lyapunov stability in probability. Finally, we rely on a numerical approach for mixed-integer optimal control for cheaply solving the problem to local optimality.

The proposed approach is rather flexible, as it easily adapts to different formulations and can explicitly account for lossy communication channels, so as to (a) exploit knowledge on the probability of having a successful communication and (b) adapt the schedule whenever a message fails to be delivered. Moreover,
since the choice of Lyapunov functions is not unique, there is some freedom in the definition of the cost which allows one to e.g. assign higher priority to some subsystem.

The rest of the paper is organised as follows. In Section~\ref{sec:formulation} we give some preliminary results and the problem formulation. In Section~\ref{sec:stability} we present our main theoretical contribution, i.e. the formulation of the communication allocation problem. 
In Section~\ref{sec:results} we evaluate our theoretical developments in a series of different scenarios. Finally, in Section~\ref{sec:conclusions} we summarise our results and discuss possible future directions.

%
%
%
%
\section{Preliminaries and Problem Formulation}
\label{sec:formulation}

We consider a network consisting of $M$ agents with unreliable communication links from the sensors to the observer and controller, and perfect communication links from the controller to the agents, as depicted in Figure~\ref{fig:scheme}.



\subsection{Definitions}
Throughout the paper, we denote agents by index $i$ and time instants by index $k$, such that each agent has state $x_{i,k}$ and control $u_{i,k}$. Whenever it is not necessary, we drop index $i$ for simplicity of notation.

We consider agents described as perturbed linear systems 
\begin{align}\label{eq:plant}
x_{k+1} = Ax_k + Bu_k + w_k,
\end{align}
where $x$, $u$ and $w$ denote the state, control and perturbation, respectively. We assume that $x_0$ and $w$ are Gaussian random variables with zero mean and covariance $X_0=\E\{x_0x_0^\top\}$ and $W=\E\{ww^\top\}$. Moreover, we assume that $x_0$ and $w$ are independent, such that $\E\{x_0w^\top\}=0$.

The observations of the states are given by
\begin{align}\label{eq:sensor}
y_k=Cx_k+v_k
\end{align}
where $v$ denotes the measurement noise, which we assume to be a Gaussian random variable with zero mean and covariance $\E\{vv^\top\}=V$. Moreover, we assume that $w$, $x_0$, $v$ are mutually independent, i.e. $\E\{wv^\top\}=\E\{x_0 v^\top\}=0$.


\subsection{Kalman Filter Updates}

Let $\delta_k =1$ if the observation is sent to the agent, $\delta_k =0$ otherwise.
The \emph{a posteriori} state estimate given by the Kalman filter~\cite{Sinopoli2004} is 
\begin{align}
\label{eq:kalman_estimate}
\hat x_{k+1} &= A \hat x_k + Bu_k + \delta_{k+1} L_{k+1} ( y_{k+1} - C(A\hat x_k +Bu_k)),
\end{align}
with 
\begin{subequations}
	\label{eq:kalman_upd}
	\begin{align}
		\bar E_{k+1} &= A E_k A^\top + W ,\\
		L_{k+1} &= \bar E_{k+1} C^\top\left ( V+C\bar E_{k+1} C^\top \right )^{-1}, \\
		E_{k+1} &= \left (I-\delta_{k+1} L_{k+1}C\right )\bar E_{k+1}, 	\label{eq:error_kalman_cov}
	\end{align}
\end{subequations}
where the estimation error is defined as $e_k=x_k-\hat x_k$ and its covariance as $E=\E\{ee^\top\}$.

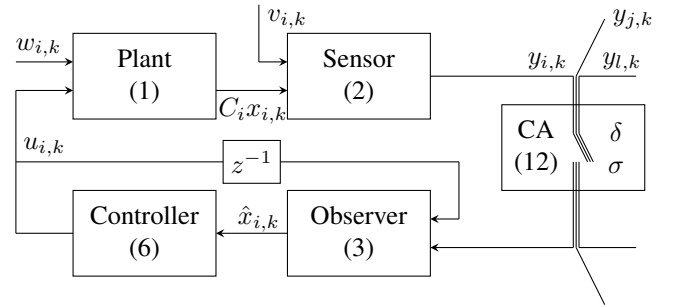
\begin{figure}[t]
	\begin{center}
		\begin{tikzpicture}[scale=0.38,>=stealth]
		\draw(0,-1.5) rectangle (5,1.5);
		\node at (2.5,0) [align=center] {Plant \\ \eqref{eq:plant}};
		\draw(7.5,-1.5) rectangle (12.5,1.5);
		\node at (10,0) [align=center] {Sensor \\ \eqref{eq:sensor}};
		\node at (6.3,-1.2) [align=left] {$C_ix_{i,k}$};
		\draw[->] (5,-0.5) --+(2.5,0);
		\draw[-] (12.5,0) --+(5,0);
		
		\draw[->] (6.5,0.5) --+(1,0);
		\draw[-] (6.5,0.5) --+(0,2);
		\node at (7.5,2) [align=left] {$v_{i,k}$};
		
		\node at (16.6,0.5) [align=left] {$y_{i,k}$};
		\draw[-] (17.5,0) --+(0,-2);
		\draw[-] (17.5,-2) --+(0.5,-1);
		\draw[-] (17.5,-3) --+(0,-3);
		
		\draw[-] (17.6,0) --+(1,2);
		\draw[-] (17.6,0) --+(0,-2);
		\draw[-] (17.6,-2) --+(0.5,-1);
		\draw[-] (17.6,-3) --+(0,-3);
		\draw[-] (17.6,-6) --+(1,-2);
		\node at (19.6,2) [align=left] {$y_{j,k}$};
		
		\draw[-] (17.7,0) --+(0,-2);
		\draw[-] (17.7,-2) --+(0.5,-1);
		\draw[-] (17.7,-3) --+(0,-3);
		\draw[-] (17.7,0) --+(2,0);
		\draw[-] (17.7,-6) --+(2,0);
		\node at (19.2,0.5) [align=left] {$y_{l,k}$};
		
		\node at (16.2,-2.5) [align=center] {CA \\ \eqref{eq:allocation_Nsa}};
		\node at (19,-2.5) [align=center] {$\delta$ \\ $\sigma$};
		\draw(15,-4) rectangle (20,-1);
		\draw[->] (17.5,-6) --+(-5,0);
		\draw(0,-7) rectangle (5,-4);
		\node at (2.5,-5.5) [align=center] {Controller \\ \eqref{eq:controller}};
		\draw(7.5,-7) rectangle (12.5,-4);
		\node at (10,-5.5) [align=center] {Observer \\ \eqref{eq:kalman_estimate}};
		\draw[<-] (5,-5.5) --+(2.5,0);
		\draw[-] (0,-5.5) --+(-2,0);
		\node at (6.5,-5) [align=left] {$\hat x_{i,k}$};
		\draw[-] (-2,-5.5) --+(0,5);
		\draw[->] (-2,-0.5) --+(2,0);
		\node at (-1,-2.5) [align=left] {$u_{i,k}$};
		
		\draw[->] (-2,0.5) --+(2,0);
		\node at (-1.2,1) [align=left] {$w_{i,k}$};
		
		\draw[-] (-2,-3) --+(7.25,0);
		\draw[-] (7.25,-3) --+(6.25,0);
		\draw(5.25,-3.75) rectangle (7.25,-2.25);
		\node at (6.25,-3) [align=left] {$z^{-1}$};
		\draw[-] (13.5,-3) --+(0,-2);
		\draw[->] (13.5,-5) --+(-1,0);
		\end{tikzpicture}
	\end{center}
	\vspace{-1.5em}
	\caption{Scheme of the system: all agents share the same communication channel, which has limited resources. The Communication Allocation (CA) algorithm assigns the channel to a selected subset of the agents.}
	\label{fig:scheme}
\end{figure}

\subsection{Imperfect Communication}

In the previous subsection we implicitly assumed a perfect downlink, i.e., all messages sent by the central node to the agents arrive at destination. While this assumption might be reasonable for physically connected systems, it is not when the communication is wireless. 
In this case, one needs to further account for the probability that the communication does not succeed.

We assume in the following that the probability $\sigma$ of having a successful communication from the sensor to the observer (see Figure~\ref{fig:scheme}) is not correlated with the estimation error $e$. This is in general not true, since the distance estimate is in general dependent on $e$. Our assumption then implicitly relies on the error $e$ being small enough. In the case of wideband communication, this requires the position estimate error to be no larger than a few meters. 
Note that $\sigma$ is typically also a function of the power of the transmission. In the following, we will assume that the distance and power are fixed and known a priori, such that $\sigma_k$ is a known time-varying signal.

By defining $s$ as a Bernoulli random variable with $\E\{s\}=\sigma$, 
%
Equation~\eqref{eq:error_kalman_cov} becomes
\begin{align}
E_{k+1}&= (I-\delta_{k+1} \sigma_k L_{k+1}C) 
\bar E_{k+1}.
\label{eq:exp_err_cov}
\end{align}
Note that $E$ defined by~\eqref{eq:exp_err_cov}, now denotes the expected covariance. 
We summarise the time evolution of 
$E$ 
as 
$E_{k+1}=f^\mathrm{E}_k(E_k,\delta_{k+1})$. This function is in general subsystem-dependent and it inherits the time-varying nature from $\sigma_k$.

\subsection{Feedback Control}

We consider the local linear feedback 
\begin{align}
	\label{eq:controller}
	u_k=-K\hat x_k.
\end{align} 
By defining $A_K\triangleq A-BK$, the state dynamics~\eqref{eq:plant} become $x_{k+1}= A_Kx_k + BK e_k + w_k$.

Note that, if we consider a quadratic cost (as we will do in Section~\ref{subsec:LQG}), then 
the so-called \textit{certainty equivalence principle} holds. It states that, the optimal solution is the same as for the corresponding deterministic problem as long as the disturbances present in the stochastic control system are zero mean \cite{bertsekas1976}.

%
%
%
%
\section{Stability with Limited Communication}\label{sec:stability}

In this section, we analyse the stability of the closed-loop system in the case of limited communication resources. 

Note that the stability of the system in terms of expected value is obtained regardless of communication, by choosing $K$ appropriately.
The time evolution of the expected values of the state and error are given respectively by
\begin{align*}
\E\{e_k\} &= 0, \\
\E\{x_{k+1}\} &= A_K\E\{x_k\}= A_K \hat x_k,
\end{align*}
where we assumed that $\hat x_0 = \E\{x_0\}$. 
Communication, however, is crucial in order to ensure that the error covariance $E$ does not grow indefinitely large.



We define the Lyapunov function of each subsystem as $\mathcal{V}(x) = x^\top P x$. 
We remark here that matrix $P$ can be easily computed by using the Lyapunov equation
\begin{align}
\label{eq:lyap}
A_K^\top P A_K - P + Q = 0,
\end{align}
with $Q \succ 0$. The freedom in choosing matrix $Q$ will later be exploited in order to tune the communication allocation as desired.


In order to select how to schedule the communication between the central node and the agents, we aim at defining the cost of communicating or not with each subsystem. In this sense, our definition resembles that VoI, already used in e.g.~\cite{Molin2015,Soleymani2016}. Because we find it relevant to relate the VoI to the stability of each subsystem, we first establish a framework aimed at building a cost directly related to the Lyapunov function of each subsystem. Second, we formulate the communication allocation problem in order to minimise the chosen cost function. Third, we establish a connection between our formulation and the definition of VoI. Finally, we prove Lyapunov stability in probability.

\subsection{Stabilising Feedback and its Cost}\label{subsec:LQG}

In this subsection, we aim at defining the cost of not communicating the state measurement or estimate. 
We express this cost based on the Lyapunov function of each subsystem in order to explicitly account for the suboptimality resulting from control actions based on imperfect estimation.
Because we consider each subsystem separately, we drop the index $i$ for simplicity.

In order to simplify the description, we first restrict our attention to the LQR case. In a second step, we will detail how, for any stabilising feedback matrix $K$, one can define LQR cost matrices that yield $K$ as optimal feedback matrix.

Consider an LQR with stage cost
\begin{align}
	\label{eq:stage_cost}
	\ell(x,u) = \matr{c}{x \\ u}^\top \matr{ll}{Q & S^\top \\ S & R}\matr{c}{x \\ u}.
\end{align}
The discrete-time algebraic Riccati equation (DARE) is
\begin{subequations}
	\label{eq:dare}
	\begin{align}
		0 &= Q - P + A^\top P A - (S^\top + A^\top P B)K, \\
		K &= (R+B^\top P B)^{-1}(S + B^\top PA).
	\end{align}
\end{subequations}
We now compare the cost of applying the feedback $u = -K\hat x$ to the cost of applying the nominal feedback $\bar u=-Kx$. We first note that
\begin{subequations}
	\begin{align}
		u &= -Kx + Ke = \bar u + K e, \\
		x_{k+1}^{\bar u} &= A_K x_k + w_k, \\
		x_{k+1}^u &= A_K x_k + BKe_k + w_k = x_{k+1}^{\bar u} + BKe_k.
	\end{align}
\end{subequations}
The effect of the estimation error over one prediction step is given by the following difference in closed-loop cost
\begin{align}
	\label{eq:lyap_diff}
	\Delta \mathcal{V}(x_k) =& \ell(x_k,-K\hat x_k) - \ell(x_k,-Kx_k) \nonumber \\
	&+ \left (x_{k+1}^{u}\right )^\top Px_{k+1}^u -\left (x_{k+1}^{\bar u}\right )^\top P x_{k+1}^{\bar u}.
\end{align}
By taking the expected value of the cost increase and noting that $\E\{e\}=0$, one obtains
\begin{align*}
	\E\{\Delta \mathcal{V}(x_k)\} &= \E\{e_k^\top K^\top RKe_k + e_k^\top K^\top B^\top PBKe_k\} \\
	&=  \mathrm{tr}\left ( K^\top\left ( R+B^\top PB\right )KE_k \right ).
\end{align*}
Because the loop is closed using the state estimate, the expected value of the increase in the Lyapunov function due to estimation error over a given time interval $[0,N]$ is then
\begin{align*}
	\sum_{k=0}^N \E\{\Delta \mathcal{V}(x_k)\} = \sum_{k=0}^N \mathrm{tr}\left ( \Gamma E_k \right ),
\end{align*}
with $\Gamma \triangleq K^\top\left ( R+B^\top PB\right )K$. Note that dynamic programming yields $\left ( R+B^\top PB\right ) \succ 0$ and, therefore, $\Gamma \succ 0$. Moreover, \eqref{eq:dare} yields $\Gamma = Q - P + A^\top P A$.


We now turn our attention to generalising the proposed framework to any stabilising feedback gain $K$.
\begin{Proposition}
	\label{prop:Kriccati}
	Given any linear system $x_{k+1}=Ax_k+Bu_k$ and stabilising feedback gain $K$, one can obtain the feedback gain $K$ as the solution of an LQR formulated using the stage cost $\ell(x,u)$ from~\eqref{eq:stage_cost} using matrices $R=I$, $S=K$, $Q=K^\top K$. 
%
\end{Proposition}
\begin{proof}
	The proposed matrices solve the DARE~\eqref{eq:dare} with $P=0$ and $R+B^\top PB \succ 0$.
	The proof is then obtained by noting that the DARE admits at most one stabilising solution, i.e. $K$, which is stabilising by assumption. 
\end{proof}
Note that the LQR obtained by applying Proposition~\ref{prop:Kriccati} has a positive semi-definite stage cost and yields $\Gamma = K^\top K$. If a positive-definite formulation is sought, the technique proposed in~\cite{Zanon2014d} can be applied. Alternatively,~\eqref{eq:lyap} can be solved and $\Gamma = Q - P + A^\top P A$ can be used.

\subsection{Optimal Communication Allocation}

After analysing each subsystem separately, we consider the aggregate cost of communication, resulting from all subsystems considered as one system. While the design of each agent's controller can be done separately, it is important to note that some degrees of freedom in the cost definition are free and should be exploited in order to define the communication cost. 
The simplest observation is that by multiplying the LQR tuning matrices by a factor $a_i \neq 0$, the same feedback gain is obtained, but the Lyapunov function matrix becomes $a_i P_i$. This allows one to assign higher or lower importance to some subsystems, thus indicating one possibility to further tune the cost for the aggregate system. Another approach can be to design the cost of the single systems in an aggregate way, i.e. by considering them as a single (sparse) system.

We remark here that one tempting approach could consist in applying~\cite[Lemma 2]{Zanon2014d}, so as to assign an arbitrarily chosen value $P_i \succ 0$ the cost-to-go matrix of each subsystem's controller. While this is always possible, the relative stage cost is typically indefinite.
Consequently, the cost-to-go is not guaranteed to be a Lyapunov function: though positive definite, it typically fails to fulfil the decrease condition.

We formulate the communication allocation problem by summarising the previously described design choices 
in matrices $\Gamma_i \triangleq K_i^\top\left ( R_i+B_i^\top P_iB_i\right )K_i$. 
The optimal allocation problem at time $j$ can then be formulated as
\begin{subequations}
	\label{eq:allocation_Nsa}
	\begin{align}
	\min_{\delta} \ \ & \sum_{i=1}^{M} \sum_{k=j}^{N+j} \mathrm{tr}(\Gamma_{i} E_{i,k}) \\
	\mathrm{s.t.} \ \ 
	& E_{i,k} = f^\mathrm{E}_{i,k}(E_{i,k-1},\delta_{i,k}),&& i\in\mathbb{I}_1^M,\, k\in\mathbb{I}_{j}^{N+j}, \label{eq:Edyn}\\ 
	& \delta_{i,k} = \{0,1\},&& i\in\mathbb{I}_1^M,\, k\in\mathbb{I}_j^{N+j}, \\
	& \sum_{i=1}^{M} \delta_{i,k} \leq \gamma,&&\phantom{i\in\mathbb{I}_1^M,\,} k\in\mathbb{I}_j^{N+j}, \label{eq:comm_limit}
	\end{align}
\end{subequations}
where we define $\mathbb{I}_a^b\triangleq\{ \ n\in\mathbb{N} \ | \ a \leq n \leq b \ \}$ and $\delta = (\delta_{1},\ldots,\delta_M)$ and $\delta_i=(\delta_{i,0},\ldots,\delta_{i,N})$. 


\subsection{The Value of Information}

In this subsection, we show the connection between our approach and the concept of \emph{value of information} (VoI), used in different settings in e.g.~\cite{Molin2015,Soleymani2016}. In the context of our setting, the VoI could be defined for agent $j$ as
\begin{subequations}
\label{eq:voi}
\begin{align}
&\nu_i(\delta_i,E_{i,j-1}) \triangleq  \sum_{k=j}^{N+j} \mathrm{tr}(\Gamma_{i} E_{i,k}) \hspace{-1em}\\
&\hspace{3em}\mathrm{s.t.} \ \ 
 E_{i,k} = f^\mathrm{E}_{i,k}(E_{i,k-1},\delta_{i,k}),&& k\in\mathbb{I}_{j}^{N+j}.
\end{align}
\end{subequations}

Then, Problem~\eqref{eq:allocation_Nsa} can be formulated in terms of VoI as
\begin{subequations}
\label{eq:allocation_Nsa2}
\begin{align}
\min_{\delta} \ \ & \sum_{i=1}^{M} \nu_i(\delta_i,E_{i,j-1}) \\
\mathrm{s.t.} \ \ 
& \delta_{i,k} = \{0,1\},&& k\in\mathbb{I}_j^{N+j}, \\
& \sum_{i=1}^{M} \delta_{i,k} \leq \gamma,&& k\in\mathbb{I}_j^{N+j}.
\end{align}
\end{subequations}
Problem~\eqref{eq:allocation_Nsa2} has the same complexity as~\eqref{eq:allocation_Nsa}. However, the problem of computing the VoI~\eqref{eq:voi} is very difficult and presents strong similarities to the computation of the cost-to-go function in dynamic programming. 

\subsection{Lyapunov Stability in Probability}

In this subsection, we analyse the conditions under which our approach yields closed-loop stability. The stability concept we will use is Lyapunov stability in probabilty~\cite{Kozin1969}. For the sake of notational simplicity, 
we consider the system as a whole, such that $x_k=(x_{1,k},\ldots,x_{M,k})$, and matrices $E$ and $P$ are defined consistently.
\begin{Definition}[Lyapunov Stability in Probability]
	Lyapunov Stability in Probability (LSP) holds for a linear system 
	if, given $P\succ0$, $\epsilon, \bar \epsilon$, there exists $\rho(\epsilon,\bar \epsilon)>0$ such that $|x_0|<\rho$ implies
	\begin{align*}
		\limsup_{k\rightarrow\infty} \mathrm{P}[x_k^\top P x_k \geq \epsilon ]\leq \bar \epsilon.
	\end{align*}
\end{Definition}


\begin{Assumption}
	\label{ass:schedule}
	There exists a baseline schedule $\delta^\mathrm{bs}$ over the time interval $[0,N]$ yielding error covariance $E_k^\mathrm{bs}$ such that
	 $\forall \ E_0 \ \in \mathcal{E} \triangleq \{ E \ | \ \mathrm{tr}(\Gamma E) \leq \mu \in \mathbb{R} \}$ it holds that $E_k^\mathrm{bs}\in\mathcal{E}$, $k=1,\ldots,N$. 
\end{Assumption}

\begin{Theorem}
	\label{thm:lsp}
	Assume that the initial state covariance, the noise covariances $W$, $V$ and the initial estimation error are finite. Suppose moreover that Assumption~\ref{ass:schedule} holds. Then, the closed-loop system with communication allocation based on the solution of Problem~\eqref{eq:allocation_Nsa} is LSP.
\end{Theorem}
\begin{proof}
	Since $E_N\in\mathcal{E}$, by applying the baseline schedule, $E_{N+1}\in\mathcal{E}$ and the problem is recursively feasible. 
	
	We apply Markov's inequality to obtain the upper bound
	\begin{align}
		\mathrm{P}[x_k^\top P x_k \geq \epsilon ]\leq \frac{\E\{x_k^\top P x_k\}}{\epsilon}.
	\end{align}
	We define $X_k \triangleq \E\{x_kx_k^\top\}$
	and write
	the time evolution of $\E\{x_k^\top P x_k\}=  \tr(P X_{k})$ as
	\begin{align*}
		\E\{x_{k+1}^\top P x_{k+1}\} =\, & \tr(P X_{k+1}) \\
		=\,&\tr(A_K^\top P A_K X_k) + \tr(P W) + \tr(\Gamma E_k).
	\end{align*}
	Since $A_K$ is stable, using~\eqref{eq:lyap} we obtain $\tr(A_K^\top P A_K X_k)=\tr( (P-Q) X_k)\leq\alpha \, \tr( P X_k) $, with $0\leq\alpha <1$.

	By optimality $\sum_{j=k}^{k+N}\mathrm{tr}( \Gamma E_j)\leq \sum_{j=k}^{k+N}\mathrm{tr}( \Gamma E_j^\mathrm{bs})$, and we conclude that  
	$$\mathrm{tr}( \Gamma E_k)\leq N\sum_{j=k}^{k+N}\mathrm{tr}( \Gamma E_j^\mathrm{bs}) \leq N\mu.$$
	This further yields the upper bound $\tr(P W) + \tr(\Gamma E_k) \leq \nu$, where $\nu \triangleq \tr(PW) + N\mu < \infty$, such that
	\begin{align*}
		\tr(P X_{k+1}) \leq \alpha \, \tr( P X_k) + \nu.
	\end{align*}
	Consequently, $\displaystyle \limsup_{k\rightarrow\infty} \, \tr( P X_k) \leq \frac{\nu}{1-\alpha}\leq \nu < \infty$. 
\end{proof}

%
%
%
%
\section{Numerical Results}
\label{sec:results}

In this section, we test the theoretical developments of the previous sections in a series of different scenarios.
We consider subsystems with 
\begin{align}
\label{eq:subsystems_example}	
A_i = \matr{cc}{1&0.1\\ 0 & 1}, && B_i = \matr{c}{0.005 \\0.1}.
\end{align}
Unless specified differently, we control each subsystem by an LQR controller designed using weighting matrices $Q_i=I$ and $R_i=0.01$. 

%
%
%
%
\subsection{Numerical Methods}
\label{sec:computations}

Problem~\eqref{eq:allocation_Nsa} is a mixed-integer optimal control problem (MIOCP) formulated in discrete time. An efficient approach, called \emph{partial outer convexification}, has been proposed in~\cite{Sager2009,Kirches2010y} to handle the numerical solution of MIOCPs. Consider a generic system with state $x$, continuous control $u$, binary control $\omega$ and system dynamics $x_+=f(x,u,\omega)$.
Partial outer convexification reformulates the dynamics as
$x_+=\sum_{j=1}^{n_\mathrm{c}} f(x,u,\omega_j)\tilde \omega_j,$
where, vectors $\omega_j$ are binary vectors spanning the space of all possible binary combinations of inputs, $n_\mathrm{c}$ is the amount of such combinations and $\tilde \omega_j$ is a newly introduced control variable. 
In order to solve the MIOCP, the integer constraint is relaxed to $\tilde \omega \in [0,1]^{n_\mathrm{c}}$ and the relaxed OCP is solved using standard techniques. This relaxation can be proven to be tighter than the relaxation of the original formulation $x_+=f(x,u,\omega)$ using $\omega \in [0,1]^{n_\mathrm{\omega}}$. 

Often the solution of the relaxed problem is integer. 
In all other cases, the continuous solution can be approximated by a switching integer solution 
by using a rounding scheme.
Note that, because the relaxed problem is nonconvex, the solver returns a local minimum. 
We refer to~\cite{Sager2009,Kirches2010y} and references therein for more details on the topic.

We remark that the problem formulations used in this paper are already in their partial outer convexification form. 
We solve the optimisation problems using Casadi~\cite{Andersson2013b} and Ipopt~\cite{Waechter2006}.

\subsection{Identical Subsystems}
We simulate first the simplest case of identical subsystems with identical noise covariance. In this simple case, the optimal solution is to communicate with subsystem $i$ only after communication has happened with all subsystems $j\neq i$.


We considered $4$ identical agents with $W_i=10^{-2}I$, $V_i= 10^{-3}I$ and a prediction horizon $N\in[1,10]$ and $\gamma=1$. Due to the symmetry of the problem, the solver returned non-integer solutions in the beginning of the simulation and a rounding strategy was necessary in the beginning. Afterwards, integer solutions were always obtained and the periodic solution $\delta_{i,k}=1$ when $k\mod i=0$, $\delta_{i,k}=0$ otherwise, was obtained (modulo phase shifts).


\subsection{Non-Identical Subsystem with Perfect Communication}

We consider now the case of non-identical subsystems. In particular, we consider different numbers of agents $M$. 

The closed-loop cost for $\gamma=1$ is displayed in Figure~\ref{fig:cl_cost_perfect_comm} for different prediction horizons $N$ and amounts of subsystems $M$. For $M<5$, horizons $N>1$ did not seem to improve the closed-loop cost. For $M\geq5$, it can be seen that, while the optimiser often falls into local minima, the cost tends to decrease by using longer prediction horizons until $N=5$. This effect is more evident when the amount $M$ of subsystems becomes larger. In most cases, the closed-loop cost is less than $5\%$ higher than the best value found by the optimiser. 
For $M\leq9$ the solver always returned an integer solution. In all other simulations the solution needed to be rounded at most twice. Simulations with $\gamma>1$ showed similar trends. Finally, the closed-loop cost dependence on $\gamma$ is shown on the bottom graph in Figure~\ref{fig:cl_cost_perfect_comm}.

\begin{figure}
	\begin{center}
		\includegraphics[width=0.5\textwidth,trim= 0 20 0 25]{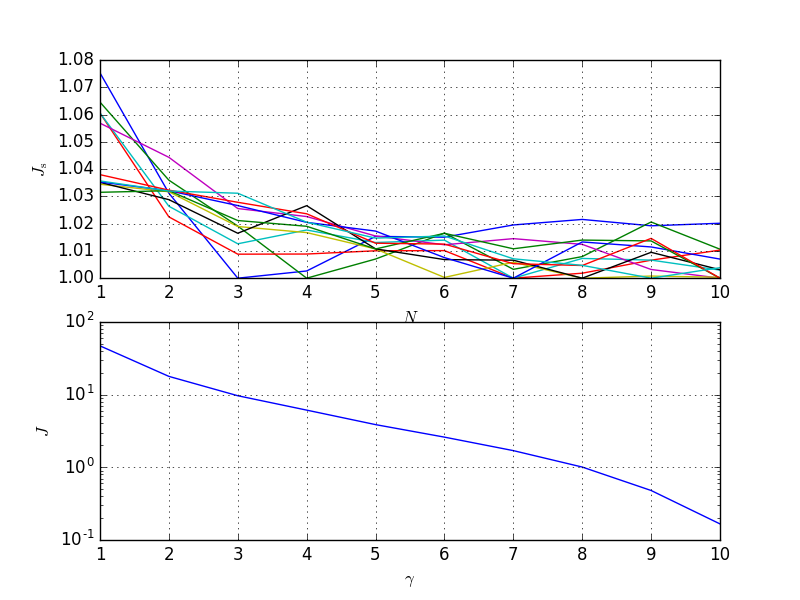}
		\caption{Top graph: closed-loop cost for $M\in[5,15], \ \gamma = 1$ as a function of the prediction horizon for several scenarios considering different numbers of agents. The value $J_\mathrm{s}$ denotes the closed-loop cost scaled around the minimum for each scenario.
		Bottom graph: Closed-loop cost $J$ for $M=10$ as a function of the amount of available communication $\gamma$ in~\eqref{eq:comm_limit}. }
		\label{fig:cl_cost_perfect_comm}
	\end{center}
\vspace{-1.5em}
\end{figure}




\subsection{Tuning the Cost of Communication}

In this subsection, we briefly present how the cost of communicating to one subsystem rather than to the others can be tuned. For the sake of simplicity, we consider two identical subsystems.

We use the following tuning for the LQR: $Q_1=I$ and $R_1=0.01$ and $Q_2=a^2Q_1$ and $R_2=a^2R_1$, where $a$ is our tuning parameter. As remarked before, this choice yields the same feedback gain for both systems, i.e. $K_1=K_2$. However, the cost-to-go matrices are $P_2=a^2P_1$, such that $\Gamma_2=a^2\Gamma_1$. As one could expect, the ratio $r=\sum_k \delta_{2,k}/\delta_{1,k}$ between the amount of communication with system $1$ and $2$ scales roughly linearly with increasing $a$. However, since the problem has an integer nature and the solvers only finds a local minimum the ratio evolves as displayed in Figure~\ref{fig:tuning}.

\begin{figure}
	\begin{center}
		\includegraphics[width=0.48\textwidth,clip,trim= 0 195 0 20]{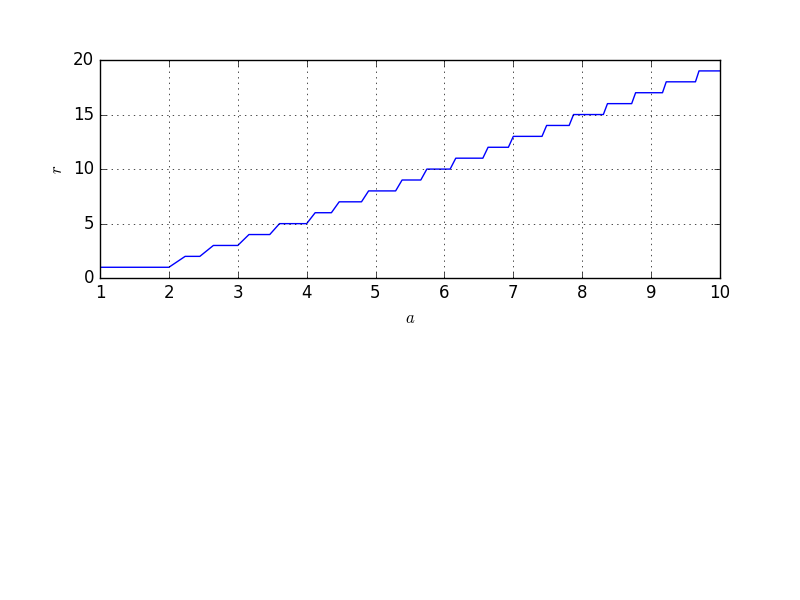}
		\caption{Ratio $r=\sum_k \delta_{2,k}/\delta_{1,k}$ as a function of parameter $a$.}
		\label{fig:tuning}
	\end{center}
\end{figure}


\subsection{Imperfect Communication}

We now consider a scenario with a lossy communication channel with probability of successful communication given by $p_{i,k} = e^{-d_{i,k}^2}$, where $d_{i,k}$ denotes the distance of agent $i$ from the central node at time $k$. We use identical systems, i.e. with variances given by
\begin{align*}
V_1 &= 10^{-3}I
, & W_1 &= 10^{-2}I
, &
V_2 &= V_1, &	W_2 &= W_1.
\end{align*}
We consider a scenario where $d_{i,k} = \cos \left ( 0.1 k + i \frac{\pi}{2} \right )$, such that when one system is at the farthest point, the other system is at the closest one. We run $100$ simulations over a horizon of $100$ steps. 
We compare the formulation explicitly accounting for lossy communication to the formulation assuming perfect communication and the baseline strategy of alternating communication between the two systems at each step, which is optimal in case of no packet loss.

We consider three scenarios, with decreasing probability of successful packet reception, as displayed in Figure~\ref{fig:prob}. The time evolution of the probability is periodic and the probability for the two systems is shifted by half a period. The resulting average closed-loop cost is displayed in FIgure~\ref{fig:cost_lossy}. We have normalised the cost with respect to the baseline strategy. For the first two scenarios, the prediction horizon does not seem to have an important role. It is interesting to note that, even if perfect communication is assumed, the algorithm uses the information on packet loss in order to improve the closed-loop cost. While in the first scenario including the information on the probability of packet loss does not improve the closed-loop cost, in the second scenario the improvement is significant. Finally, in the third and most challenging scenario, explicitly accounting for the probability of packet drop is crucial for performance, as not doing so significantly deteriorates the performance with respect to the baseline strategy. Note that in this last scenario, longer prediction horizons improve the performance. For prediction horizons $N>15$, i.e. longer than approximately half a period of the package drop probability variation, the performance does not improve further.

\begin{figure}
	\begin{center}
		\includegraphics[width=0.48\textwidth,clip,trim= 0 195 0 0]{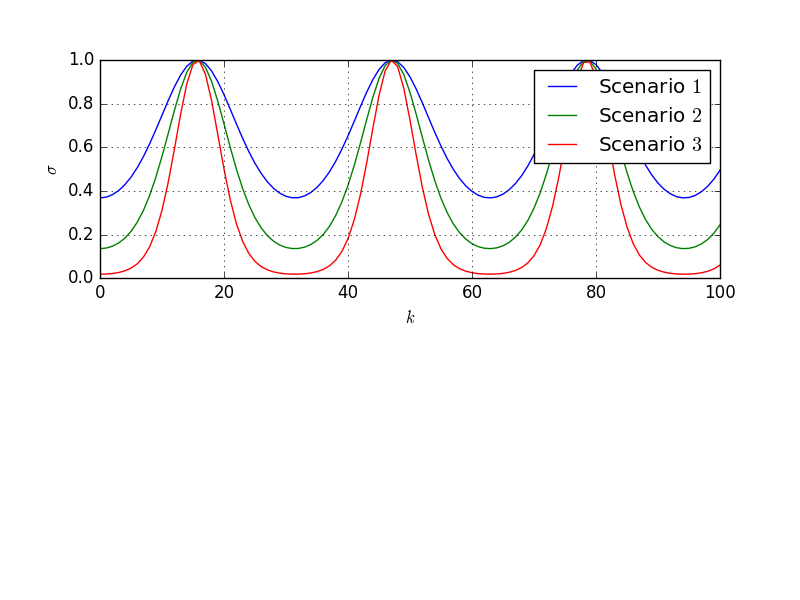}
		\caption{Periodic probability of having a successful communication for system $1$. System $2$ has the same probability time evolution but shifted by half a period.}
		\label{fig:prob}
	\end{center}
\end{figure}

%
%
%
%
\section{Conclusions and Future Research}
\label{sec:conclusions}

We have presented an approach for allocating communication between sensors and agents in the case of limited communication resources. The problem formulation minimises the Lyapunov function increase caused by imperfect communication and can incorporate knowledge on packet loss probability in order to improve the closed-loop cost.

In the case of perfect communication, some closed-loop solutions are periodic. This happens especially for small $M$, while for large amounts of subsystems this periodic behaviour does not appear. While it seems hard to prove that a periodic behaviour is optimal, the fact that the optimiser does not return a periodic solution could be explained that, due to nonconvexity, only locally optimal solutions are obtained. If periodic solutions are sought, the initial constraint in Problem~\eqref{eq:allocation_Nsa} can be replaced by a periodicity constraint and the problem can be solved for different period lengths in order to select the solution with the lowest cost. 

Moreover, in a fully distributed setting, the proposed approach can be used in combination with the approach of~\cite{Mamduhi2015}: the offline deployment of our framework can give indications on how to select the tuning parameters of that algorithm, i.e. the quadratic cost the relative threshold.

\begin{figure}
	\begin{center}
		\includegraphics[width=0.48\textwidth]{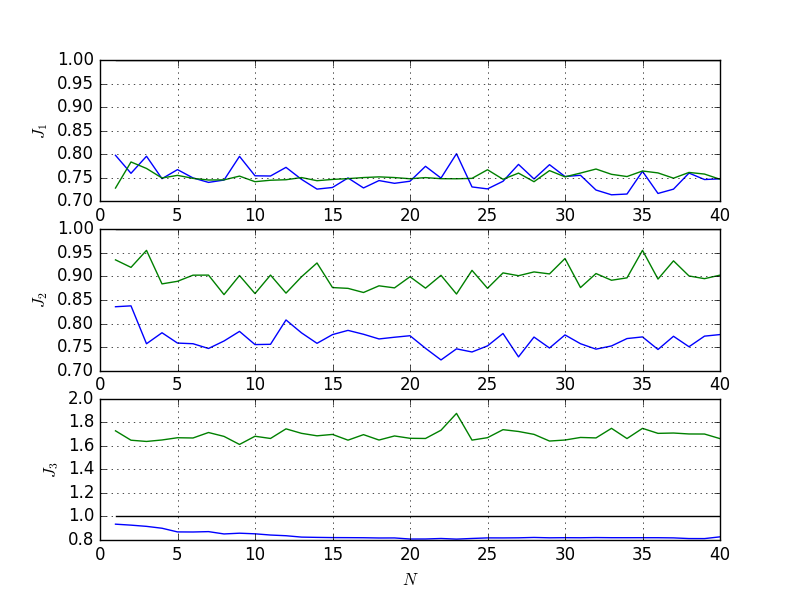}
		\caption{Average closed-loop cost obtained by assuming perfect communication (green) and by explicitly accounting for the packet drop probability (blue) as a function of the prediction horizon $N$ normalised with respect to the cost obtained with the baseline strategy (black). Top graph: scenario 1. Middle graph: scenario 2. Bottom graph: scenario 3. }
		\label{fig:cost_lossy}
	\end{center}
\end{figure}

\bibliographystyle{IEEEtran}
\bibliography{bibliography}

\begin{thebibliography}{10}
\providecommand{\url}[1]{#1}
\csname url@samestyle\endcsname
\providecommand{\newblock}{\relax}
\providecommand{\bibinfo}[2]{#2}
\providecommand{\BIBentrySTDinterwordspacing}{\spaceskip=0pt\relax}
\providecommand{\BIBentryALTinterwordstretchfactor}{4}
\providecommand{\BIBentryALTinterwordspacing}{\spaceskip=\fontdimen2\font plus
\BIBentryALTinterwordstretchfactor\fontdimen3\font minus
  \fontdimen4\font\relax}
\providecommand{\BIBforeignlanguage}[2]{{%
\expandafter\ifx\csname l@#1\endcsname\relax
\typeout{** WARNING: IEEEtran.bst: No hyphenation pattern has been}%
\typeout{** loaded for the language `#1'. Using the pattern for}%
\typeout{** the default language instead.}%
\else
\language=\csname l@#1\endcsname
\fi
#2}}
\providecommand{\BIBdecl}{\relax}
\BIBdecl

\bibitem{Gupta2010}
R.~A. Gupta and M.~Y. Chow, ``{Networked Control System: Overview and Research
  Trends},'' \emph{IEEE Transactions on Industrial Electronics}, vol.~57,
  no.~7, pp. 2527--2535, July 2010.

\bibitem{Zhang2013}
L.~Zhang, H.~Gao, and O.~Kaynak, ``{Network-Induced Constraints in Networked
  Control Systems -- A Survey},'' \emph{IEEE Transactions on Industrial
  Informatics}, vol.~9, no.~1, pp. 403--416, Feb. 2013.

\bibitem{Molin2015}
A.~Molin, C.~Ramesh, H.~Esen, and K.~H. Johansson, ``Innovations-based priority
  assignment for control over can-like networks,'' in \emph{2015 54th IEEE
  Conference on Decision and Control (CDC)}, Dec 2015, pp. 4163--4169.

\bibitem{Mamduhi2015}
M.~Mamduhi, D.~Tolic, and S.~Hirche, ``Decentralized event-based scheduling for
  shared-resource networked control systems,'' in \emph{14th Annual European
  Control Conference (ECC)}, Jul 2015.

\bibitem{Soleymani2016}
T.~Soleymani, S.~Hirche, and J.~S. Baras, ``Optimal self-driven sampling for
  estimation based on value of information,'' in \emph{2016 13th International
  Workshop on Discrete Event Systems (WODES)}, May 2016, pp. 183--188.

\bibitem{Tomlin2014}
L.~Zhao, W.~Zhang, J.~Hu, A.~Abate, and C.~J. Tomlin, ``{On the Optimal
  Solutions of the Infinite-Horizon Linear Sensor Scheduling Problem},''
  \emph{IEEE Transactions on Automatic Control}, vol.~59, no.~10, pp.
  2825--2830, Oct 2014.

\bibitem{Sinopoli2014}
Y.~Mo, E.~Garone, and B.~Sinopoli, ``On infinite-horizon sensor scheduling,''
  \emph{Systems \& Control Letters}, vol.~67, pp. 65--70, 2014.

\bibitem{Vitus2012}
M.~P. Vitus, W.~Zhang, A.~Abate, J.~Hu, and C.~J. Tomlin, ``On efficient sensor
  scheduling for linear dynamical systems,'' \emph{Automatica}, vol.~48,
  no.~10, pp. 2482--2493, 2012.

\bibitem{Sinopoli2004}
B.~Sinopoli, L.~Schenato, M.~Franceschetti, K.~Poolla, M.~I. Jordan, and S.~S.
  Sastry, ``Kalman filtering with intermittent observations,'' \emph{IEEE
  Transactions on Automatic Control}, vol.~49, no.~9, pp. 1453--1464, Sept
  2004.

\bibitem{bertsekas1976}
D.~Bertsekas, \emph{Dynamic programming and stochastic control}.\hskip 1em plus
  0.5em minus 0.4em\relax Academic Press, 1976.

\bibitem{Zanon2014d}
M.~Zanon, S.~Gros, and M.~Diehl, ``{I}ndefinite {L}inear {MPC} and
  {A}pproximated {E}conomic {MPC} for {N}onlinear {S}ystems,'' \emph{Journal of
  Process Control}, vol.~24, pp. 1273--1281, 2014.

\bibitem{Kozin1969}
F.~Kozin, ``A survey of stability of stochastic systems,'' \emph{Automatica},
  vol.~5, no.~1, pp. 95--112, Jan. 1969.

\bibitem{Sager2009}
S.~Sager, G.~Reinelt, and H.~Bock, ``{D}irect {M}ethods {W}ith {M}aximal
  {L}ower {B}ound for {M}ixed-{I}nteger {O}ptimal {C}ontrol {P}roblems,''
  \emph{Mathematical Programming}, vol. 118, no.~1, pp. 109--149, 2009.

\bibitem{Kirches2010y}
C.~Kirches, ``Fast numerical methods for mixed-integer nonlinear
  model-predictive control,'' Ph.D. dissertation, Ruprecht-Karls-Universit\"at
  Heidelberg, 2010.

\bibitem{Andersson2013b}
J.~Andersson, ``{A} {G}eneral-{P}urpose {S}oftware {F}ramework for {D}ynamic
  {O}ptimization,'' {P}h{D} thesis, Arenberg Doctoral School, KU Leuven,
  October 2013.

\bibitem{Waechter2006}
A.~W\"achter and L.~Biegler, ``{O}n the {I}mplementation of a {P}rimal-{D}ual
  {I}nterior {P}oint {F}ilter {L}ine {S}earch {A}lgorithm for {L}arge-{S}cale
  {N}onlinear {P}rogramming,'' \emph{Mathematical Programming}, vol. 106,
  no.~1, pp. 25--57, 2006.

\end{thebibliography}
%
%
%
%
\end{document}